\documentclass[a4paper, 11pt]{amsart}
\usepackage{graphicx, tikz, euscript, enumitem, kantlipsum}
\usetikzlibrary{decorations.pathreplacing,angles,quotes}

\usepackage{amssymb}
\usepackage{amsmath}
\usepackage{amsthm}
\usepackage{amscd}
\usepackage{mathbbol}
\usepackage[all,2cell]{xy}
\usepackage{mathtools}
\usepackage{tikz}
\usepackage{tikz-cd}
\usetikzlibrary{calc}

\usetikzlibrary{arrows}

\usepackage[colorlinks]{hyperref}

\usepackage{tikz}
\usetikzlibrary{arrows,backgrounds,shadows,decorations.markings,matrix}

\usepackage[margin=1in]{geometry}

\newcommand{\ko}{\: , \;}

\UseAllTwocells \SilentMatrices
\newtheorem{thm}{Theorem}[section]

\newtheorem{cor}[thm]{Corollary}

\newtheorem{exm}[thm]{Example}

\newtheorem{prop}[thm]{Proposition}
\theoremstyle{definition}
\newtheorem{defn}[thm]{Definition}
\theoremstyle{remark}
\newtheorem{rem}[thm]{\bf Remark}
\numberwithin{equation}{section}

\newcommand{\Ext}{\mathrm{Ext}}

\newcommand{\HH}{\mathrm{HH}}

\newcommand{\Hom}{\mathrm{Hom}}

\newcommand{\op}{\mathrm{op}}

\newcommand{\iso}{\xrightarrow{_\sim}}
\newcommand{\liso}{\xleftarrow{_\sim}}

\makeatletter
\newcommand{\smallotimes}{\mathbin{\mathpalette\make@small\otimes}}

\newcommand{\make@small}[2]{%
  \vcenter{\hbox{%
    $\m@th\ifx#1\displaystyle\scriptstyle\else\ifx#1\textstyle\scriptstyle
     \else\scriptscriptstyle\fi\fi#2$%
  }}%
}
\makeatother

\begin{document}

\title[$A_\infty$-deformations of zigzag algebras via Ginzburg dg algebras]{$A_\infty$-deformations of zigzag algebras\\[0.15cm] via Ginzburg dg algebras}

\author{Junyang Liu and Zhengfang Wang}

\subjclass[2020]{16E05, 13D03, 16G20, 16E60}
\thanks{}
\keywords{$A_\infty$-deformation, zigzag algebra, Ginzburg dg algebra, Hochschild cohomology, preprojective algebra}

\date{\today}

\begin{abstract}
This note aims to give a short proof of the recent result due to Etg\"u--Lekili (2017) and Lekili--Ueda (2021): the zigzag algebra of any finite tree over a field of characteristic $0$ is intrinsically formal if and only if the tree is of type $ADE$. We also complete the proof of this result by considering a field of arbitrary characteristic for type $E$, which was still open.
\end{abstract}

\maketitle

\section{Introduction}
The zigzag algebras associated with graphs are important algebras which naturally appear in various areas such as representation theory, symplectic geometry and categorification; see e.g.\ \cite{HK, EK, ST, EL, MT}. These algebras have many nice features. For instance, they are symmetric and related to Ginzburg dg algebras by derived Koszul duality; see \cite{HK} for more properties.  

In this short note, we provide an algebraic proof of the following result due to \cite{EL} and \cite{LU}. 
\begin{thm}\label{thm:ET}
Let $\Bbbk$ be a field and $\Gamma$ a finite tree. Then the zigzag algebra $Z(\Gamma)$, which is graded by path length,  is intrinsically formal if and only if $\Gamma$ is of type ADE and the characteristic of $\Bbbk$ is not $\mathrm{bad}$ ($2$ for type $D$, $2$ or $3$ for $E_6$ and $E_7$, and $2$, $3$ or $5$ for $E_8$). 
\end{thm}
We mention that this formality result plays an essential role in $4$-dimensional symplectic geometry as in \cite{EL}. The result was proved for types $A$ and $D$ in \cite[Thm.~40]{EL} and \cite[Thm.~44]{EL} (also see \cite[Lem.~4.21]{ST} for type $A$).
The type $E$ was conjectured in \cite[Conj.~19]{EL} and verified later in \cite[\S 5]{LU} when the characteristic of $\Bbbk$ is $0$ by a geometric approach based on the computation of the Hochschild cohomology of matrix factorisations \cite{Dyc}. We complete the proof of Theorem~\ref{thm:ET} by considering a field $\Bbbk$ of arbitrary characteristic for type $E$. 

Our proof is purely algebraic and relatively direct. First, by the derived Koszul duality between  $Z(\Gamma)$ and the $2$-dimensional Ginzburg dg algebra $\Pi_2(Q)$, where $Q$ is a quiver whose underlying graph is $\Gamma$, there is an isomorphism  
\[
\begin{tikzcd}
\HH^{p,q}(Z(\Gamma), Z(\Gamma)) \arrow[no head]{r}{\sim} & \HH^{p, q}(\Pi_2(Q), \Pi_2(Q))
\end{tikzcd}
\]
between the bigraded Hochschild cohomologies for all integers $p$ and $q$. Here both $Z(\Gamma)$ and $\Pi_2(Q)$ are viewed as differential bigraded algebras; see Proposition \ref{prop:Zigginzburghh}.

Secondly, we show in Proposition \ref{prop:ginzburghh} that we have
\[
\begin{tikzcd}
\HH^{2, q}(\Pi_2(Q), \Pi_2(Q)) \arrow{r}{\sim} & (\Lambda_Q/[\Lambda_Q, \Lambda_Q])^{q+2}
\end{tikzcd}
\]
for all integers $q$, where $\Lambda_Q$ is the preprojective algebra of $Q$; see \cite{GP}.  This isomorphism is the key ingredient in our proof, which  may be viewed as a bigraded version of Van den Bergh duality; see Remark \ref{remark:vandenberghduality}.  Combining the above two isomorphisms we obtain 
\[
\begin{tikzcd}
\HH^{2, q}(Z(\Gamma), Z(\Gamma)) \arrow[no head]{r}{\sim} & (\Lambda_Q/[\Lambda_Q, \Lambda_Q])^{q+2}
\end{tikzcd}
\]
for all integers $q$. Then Theorem \ref{thm:ET} follows by using the criterion for intrinsic formality (Theorem \ref{thm:KST})  and the description of $\Lambda_Q/[\Lambda_Q, \Lambda_Q]$ due to \cite[Thm.~13.1.1]{Sc} and \cite[Thm.~1.4]{MOV} which involves the bad characteristics; see Corollary \ref{corollary:HH}. 

\section{Bigraded Hochschild cohomology} 
As we will use bigradings of zigzag algebras and $2$-dimensional Ginzburg dg algebras and consider the bigraded Hochschild cohomology, let us fix notations for differential bigraded algebras; see \cite{Kel03}.

A {\it differential bigraded vector space} $(V, d)$ is a bigraded vector space $V = \bigoplus_{p, q\in \mathbb Z} V^{p, q}$ together with a differential $d$ of bidegree $(1, 0)$, where the first grading denotes the cohomological grading and the second one denotes the Adams grading which is viewed as an additional grading. Notice that we obtain a complex $V^{\bullet, q}$ for each fixed integer $q$. 

For any two differential bigraded vector spaces $(U, d_U)$ and $(V, d_V)$, we may consider the differential bigraded $\Hom$-space $\Hom^{\bullet, \bullet}(U, V) = \bigoplus_{p,q\in \mathbb Z} \Hom^{p,q}(U, V)$ and the tensor product $U \otimes V= \bigoplus_{p,q\in\mathbb Z}(U \otimes V)^{p, q}$ with the usual differentials, where 
\[
\Hom^{p,q}(U, V) = \prod_{i,j\in \mathbb Z} \Hom_\Bbbk(U^{i,j}, V^{p+i, q+j}) \ko (U \otimes V)^{p,q} = \bigoplus_{i,j\in \mathbb Z} U^{i,j} \otimes_\Bbbk V^{p-i,q-j}\: .
\]
Let $(A, d)$ be a differential bigraded algebra (whose multiplication is of bidegree $(0,0)$). As usual, we may define differential bigraded $A$-modules (whose differentials are of bidegree $(1,0)$) and the associated derived category. Then, for any two differential bigraded $A$-modules $M$ and $N$, the derived Hom-space $\mathbb{R}\mathrm{Hom}^{\bullet,\bullet}_A(M,N)$ carries a bigrading such that we have
\[
\mathrm{H}^p(\mathbb{R}\mathrm{Hom}^{\bullet, q}_A(M, N)) = \Ext^{p,q}_A(M, N)
\]
for all integers $p$ and $q$. Similar to the usual case, the bigraded  Hochschild cohomology $\HH^{p,q}(A, A)=\Ext_{A^e}^{p, q}(A, A)$ may be computed by the bigraded Hochschild cochain complex $C^{\bullet, \bullet}(A, A)$ with the usual differential, where
\[
C^{p,q}(A, A)=\prod_{i\geq 0} \Hom^{p-i,q}(A^{\otimes i}, A)\: .
\]
Namely, we have
\[
\HH^{p,q}(A, A)=\mathrm{H}^p(C^{\bullet, q}(A, A))\: .
\]
Here we denote the enveloping algebra of $A$ by $A^e= A \otimes A^{\op}$, which carries a natural bigrading. 

\section{$2$-dimensional Ginzburg dg algebras and Hochschild cohomology}
\subsection{$2$-dimensional Ginzburg dg algebras}
Let $\Bbbk$ be a field of any characteristic and $Q$ a finite connected quiver. Denote by $\widetilde Q$ the double quiver obtained from $Q$ by adding an arrow $\alpha^*\colon j \to i$ for each arrow $\alpha\colon i \to j$. Denote by $\overline Q$ the quiver obtained from $\widetilde Q$ by additionally adding a loop $t_i$ at each vertex $i \in Q_0$. We view $\overline Q$ as a graded quiver by assigning each arrow in $\widetilde Q$ degree $0$ and each loop $t_i$ degree $-1$. The {\it $2$-dimensional (non-complete) Ginzburg dg algebra} $\Pi_2(Q)$ by definition is the dg path algebra $(\Bbbk \overline Q, d)$, where the differential $d$ is determined by
\[
d(\alpha) = 0 = d(\alpha^*)\text{ for all $\alpha \in Q_1$}\quad \mbox{and}\quad d(t) = \sum_{\alpha\in Q_1} [\alpha, \alpha^*]\ko \text{where }t = \sum_{i\in Q_0} t_i \: .
\]
Then we have $d(t_i) = e_i d(t) e_i$ since $d(e_i) = 0$ for all $i \in Q_0$.

\subsection{A small cofibrant resolution}
Let us simply denote the $2$-dimensional Ginzburg dg algebra $\Pi_2(Q)$ by $B$. We also denote $\otimes = \otimes_{\Bbbk Q_0}$ and $\Hom = \Hom_{(\Bbbk Q_0)^e}$ from now on. 

We recall from \cite{Kel11} a small cofibrant resolution of $B$. First, we have a cofibrant dg $B$-bimodule $(B \otimes \Bbbk \overline Q_1\otimes B, \delta)$, where the differential $\delta$ is defined as the composition
\[
\begin{tikzcd}
B \otimes \Bbbk \overline Q_1 \otimes B \arrow{r}{d'} & B \otimes B \otimes B \arrow{r}{\rho} & B \otimes \Bbbk \overline Q_1 \otimes B \: .
\end{tikzcd}
\]
Here $\overline Q_1$ is the set of arrows in the quiver $\overline Q$ and the two maps $d'$ and $\rho$ are determined by
\begin{align*}
d'(a \otimes x \otimes b ) &= (-1)^{|a|} a d(x) b \quad \mbox{and} \\
\rho(a \otimes x_1x_2\ldots x_p \otimes b) &=\sum_{i=1}^p a x_1 \ldots x_{i-1} \otimes x_i \otimes x_{i+1} \ldots x_p b
\end{align*}
for all arrows $x$ in $\overline Q$, paths $x_1 x_2 \ldots x_p$ in $\overline Q$, and $a$, $b\in B$.

\begin{prop}[{\cite[Proposition 3.7]{Kel11}}]\label{prop:resolution}
The mapping cone $\mathrm{Cone}(\theta)$ of the map
\[
\theta \colon (B \otimes \Bbbk \overline Q_1\otimes B, \delta) \longrightarrow B\otimes B
\]
given by $\theta(a \otimes x \otimes b) = ax \otimes b - a \otimes xb$
is a cofibrant resolution of $B$ as a dg $B$-bimodule.
\end{prop}

\subsection{Hochschild cohomology of $2$-dimensional Ginzburg dg algebras}\label{subsection:hochschildginzburg}
We apply the resolution in Proposition \ref{prop:resolution} to compute the bigraded Hochschild cohomology $\HH^{2, q}(B, B)$. 

For this, we will view $B$ as a differential bigraded algebra by assigning each arrow in the double quiver $\widetilde Q$ the bidegree $(0, 1)$ and each loop $t_i$ the bidegree  $(-1, 2)$. Then each path admits a bidegree by the multiplication which is of bidegree $(0,0)$. Clearly,  the differential $d$ is of bidegree $(1,0)$. Notice that the cofibrant resolution $\mathrm{Cone}(\theta)$ in Proposition \ref{prop:resolution} is compatible with the bigrading, so it can be used to compute the bigraded Hochschild cohomology.  

The following result is the key observation of this notice. 
\begin{prop}\label{prop:ginzburghh}
For any integer $q$, we have
\[
\HH^{2,q}(B, B) \xlongrightarrow{_\sim} (\Lambda_Q/[\Lambda_Q, \Lambda_Q])^{q+2}\: ,
\]
where $\Lambda_Q=\Bbbk \widetilde Q/\sum_{\alpha \in Q_1} [\alpha, \alpha^*]$ is the preprojective algebra of $Q$.
\end{prop}
\begin{proof}
Notice that for any integers $p$ and $q$, we have 
\[
\begin{tikzcd}
\Hom^{p, q}_{B^e}(\mathrm{Cone}(\theta), B) \arrow{r}{\sim} & \Hom^{p, q}(\Bbbk Q_0, B) \oplus \Hom^{p-1, q}(\Bbbk \overline Q_1, B)\: ,
\end{tikzcd}
\]
where we use the natural isomorphism
\[
\begin{tikzcd}
\Hom^{p, q}_{B^e}(B\otimes X\otimes B, B) \arrow{r}{\sim} & \Hom^{p, q}(X, B)
\end{tikzcd}
\]
for $X = \Bbbk Q_0$ and $X= \Bbbk \overline Q_1$. Recall that we denote $\otimes = \otimes_{\Bbbk Q_0}$ and $\Hom=\Hom_{(\Bbbk Q_0)^e}$. Clearly, we have the following decomposition of $\Bbbk Q_0$-bimodule (i.e.\ $(\Bbbk Q_0)^e$-module)
\[
\Bbbk \overline Q_1 = \Bbbk \widetilde Q_1 \oplus \Bbbk Q_0 t\: ,
\]
where notice that $t = \sum_{i\in Q_0} t_i$ is $\Bbbk Q_0$-central. It yields the following decomposition
\begin{align*}
\Hom^{p-1, q}(\Bbbk \overline Q_1, B) & \iso \Hom^{p-1, q}(\Bbbk \widetilde Q_1, B) \oplus  \Hom^{p-1, q}(\Bbbk Q_0 t, B)\\
& \simeq \Hom(\Bbbk \widetilde Q_1, B^{p-1,q+1}) \oplus \Hom(\Bbbk Q_0, B^{p-2, q+2}),
\end{align*}
where the second isomorphism follows since $\Bbbk\widetilde Q_1$ is concentrated in bidegree $(0,1)$ and $\Bbbk Q_0 t$ is concentrated in bidegree $(-1,2)$.

Since $B$ is concentrated in (cohomological) non-positive degrees, it follows that we have
\[
 \Hom^{\geq 1, q}(\Bbbk Q_0, B) = 0 \quad \text{and} \quad \Hom^{\geq 2, q}(\Bbbk \overline Q_1, B) = 0\: .
\]
Therefore, for any fixed integer $q$, the complex $\Hom^{\bullet, q}_{B^e}(\mathrm{Cone}(\theta), B)$ in degrees $1$, $2$, $3$ is given by
\[
\begin{tikzcd}[ampersand replacement=\&]
\Hom^{1, q}_{B^e} \arrow{rr}\arrow{d}{\wr} \& \& \Hom^{2, q}_{B^e}\arrow{r}\arrow{d}{\wr} \& \Hom^{3, q}_{B^e}\arrow{d}{\wr} \\
\Hom(\Bbbk\widetilde{Q}_1, B^{0,q+1}) \oplus \bigoplus_{i \in Q_0} e_i B^{-1,q+2} e_i\arrow{rr}{\partial=
\begin{bsmallmatrix}
\partial_1 & \partial_2
\end{bsmallmatrix}
} \& \& \bigoplus_{i \in Q_0} e_i B^{0, q+2} e_i\arrow{r} \& 0\mathrlap{\: ,}
\end{tikzcd}
\]
where we denote $\Hom^{p, q}_{B^e}=\Hom^{p, q}_{B^e}(\mathrm{Cone}(\theta), B)$ and use the  isomorphism
\[
\Hom(\Bbbk Q_0, B^{p,q}) \xlongrightarrow{_\sim} \bigoplus_{i \in Q_0} e_i B^{p,q} e_i \: .
\]
The map $\partial = 
\begin{bmatrix}
\partial_1 & \partial_2
\end{bmatrix}$
is given as follows. For any $f \in \Hom(\Bbbk \widetilde Q_1, B^{0,q+1})$,  we have 
\begin{align}\label{align:partial1}
\partial_1(f) = \sum_{\alpha\in Q_1} \left( f(\alpha)\alpha^* +\alpha f(\alpha^*) - f(\alpha^*) \alpha - \alpha^* f(\alpha)\right) = \sum_{\alpha \in Q_1} \left([f(\alpha), \alpha^*] + [ \alpha, f(\alpha^*)]\right)\: .
\end{align}
For any $ p \in  \bigoplus_{i \in Q_0} e_i B^{-1, q+2} e_i$, we have $\partial_2(p) = d(p)$, where $d$ is the differential of $B$. 
 
Notice that we have $\HH^{2, q}(B, B) \iso \mathrm{Coker}(\partial)$. By the isomorphism $\mathrm{H}^{0}(B^{\bullet, q})\iso \Lambda_Q^q$, see e.g.\ \cite[Thm.~7]{EL}, we infer that we have
\[
\begin{tikzcd}
\mathrm{Coker}(\partial_2) \arrow[no head]{r}{\sim} & \bigoplus_{i\in Q_0} e_i\mathrm{H}^{0}(B^{\bullet, q+2})e_i \arrow{r}{\sim} & \bigoplus_{i\in Q_0} e_i\Lambda_Q^{q+2}e_i \: ,
\end{tikzcd}
\]
where $\bigoplus_{i\in Q_0} e_i\Lambda_Q^{q+2}e_i$ is spanned by the cycles of length $q+2$.  Then the map $\partial_1$ induces
\[
 \overline{\partial_1} \colon \Hom(\Bbbk \widetilde Q_1, B^{0,q+1}) \longrightarrow \bigoplus_{i\in Q_0} e_i\Lambda_Q^{q+2}e_i \: .
 \]

Clearly, we have $\mathrm{Coker}(\partial)\iso \mathrm{Coker}(\overline{\partial_1})$ and thus
\[
\begin{tikzcd}
\HH^{2, q}(B, B) \arrow{r}{\sim} & \mathrm{Coker}(\overline{\partial_1}) \arrow{r}{\sim} & (\Lambda_Q/[\Lambda_Q, \Lambda_Q])^{q+2}\: .
\end{tikzcd}
\]
Here let us explain the last isomorphism. We have the natural map (induced by the inclusion)
\[
\varphi \colon \bigoplus_{i\in Q_0} e_i\Lambda_Q^{q+2}e_i \longrightarrow (\Lambda_Q/[\Lambda_Q, \Lambda_Q])^{q+2}\: .
\]
Notice that $\varphi$ is surjective since any non-cyclic path in $\Lambda_Q$ belongs to $[\Lambda_Q, \Lambda_Q]$. By the equality~\eqref{align:partial1}, we have $\varphi \circ \overline{\partial_1} = 0$. Thus, it suffices to show that we have $\mathrm{Ker}(\varphi) \subseteq \mathrm{Im}(\overline{\partial_1})$. For this, notice that for any nontrivial cycle $x_1x_2\ldots x_{p+q} \in \Lambda_Q$ with $x_i\in \widetilde Q_1$, the following identity
\begin{align}\label{align:bracket}
[x_1\ldots x_p, x_{p+1}\ldots x_{p+q}] = [x_1, x_2\ldots x_{p+q}] + \cdots + [x_{p}, x_{p+1}\ldots x_{p+q}x_1\ldots x_{p-1}]
\end{align}
holds. Clearly, $\mathrm{Ker}(\varphi)$ is spanned by commutators of the form on the left hand side of the equality~\eqref{align:bracket}. We claim that each term on the right hand side of the equality~\eqref{align:bracket} lies in $\mathrm{Im}(\overline{\partial_1})$. Indeed, for each $1\leq i \leq p$, we define an element $f_i\in \Hom(\Bbbk \widetilde Q_1, B^{0,q+1})$  as follows. If $x_i$ equals $\alpha$ for some $\alpha \in Q_1$, then we define
$$f_i(\alpha^*)=x_{i+1}\ldots x_{p+q}x_1 \ldots x_{i-1}$$
and $f_i(\beta)=0$ for all $\beta \neq \alpha^*$ in $\widetilde Q_1$. If $x_i$ equals $\alpha^*$ for some $\alpha \in Q_1$, then we define  $$f_i(\alpha)=-x_{i+1}\ldots x_{p+q}x_1 \ldots x_{i-1}$$ and $f_i(\beta)=0$ for all $\beta \neq \alpha$ in $\widetilde Q_1$. Then by the equality~\eqref{align:partial1}, we have $$\overline{\partial_1}(f_i)=[x_i,x_{i+1}\ldots x_{p+q}x_1 \ldots x_{i-1}].$$ This proves the claim. So $\mathrm{Ker}(\varphi)$ is contained in $\mathrm{Im}(\overline{\partial_1})$.
\end{proof}

\begin{rem}\label{remark:vandenberghduality}
This remark is due to Bernhard Keller. Let $B$ be a connective and left $d$-Calabi--Yau dg algebra (in our case, $d$ equals $2$). Then we have isomorphisms 
\[
\HH^d(B, B) \liso \mathrm{H}^d(B \otimes_{B^e}^{\mathbb L} \mathbb{R}\mathrm{Hom}_{B^e}(B, B^e))\iso \mathrm{H}^0(B \otimes_{B^e}^{\mathbb L} B) \simeq \mathrm{H}^0(B) \otimes_{(\mathrm{H}^0(B))^e} \mathrm{H}^0(B)\: ,
\]
where the first isomorphism uses the smoothness of $B$, the second uses the Calabi--Yau property (i.e.\ $\mathbb{R}\mathrm{Hom}_{B^e}(B, B^e)[d] \iso B$), and the third uses the connectiveness (i.e. $\mathrm{H}^{>0}(B) = 0$).

This argument gives a proof of the isomorphism in Proposition~\ref{prop:ginzburghh} after forgetting the Adams grading without using the small cofibrant resolution. The isomorphism in Proposition \ref{prop:ginzburghh} may be viewed as a bigraded version of {\it Van den Bergh duality} \cite{VdB}. The bigrading is useful in order to write explicitly dg deformations of $B$ corresponding to elements of $\HH^{2,q}(B, B)$ in section~\ref{subsection4.3}.
\end{rem}

\begin{cor}\label{corollary:HH}Let $Q$ be any finite connected quiver and $B$ the $2$-dimensional Ginzburg dg algebra $\Pi_2(Q)$.  
Then we have $\HH^{2, q}(B, B) = 0$ for all positive integers $q$ if and only if the underlying graph of $Q$ is of type $ADE$ and the field $\Bbbk$ is not of bad characteristic.  \end{cor}
\begin{proof}
 If the underlying graph of $Q$ is of type $ADE$, then this follows from Proposition \ref{prop:ginzburghh} and \cite[Thm.~13.1.1]{Sc} which involves the bad characteristics. Otherwise, the quiver $Q$ contains a subquiver $Q'$ whose underlying graph is of extended $ADE$ type, so that the map \linebreak $\Lambda_Q/[\Lambda_Q, \Lambda_Q]\to \Lambda_{Q'}/[\Lambda_{Q'}, \Lambda_{Q'}]$ is surjective. By \cite[Thm.~13.1.1]{Sc} again, the latter is always infinite-dimensional, so is $\HH^{2, >0}(B, B)$ by Proposition \ref{prop:ginzburghh} for any field $\Bbbk$.
\end{proof}

\section{$A_\infty$-deformations of zigzag algebras} 
\subsection{$A_\infty$-deformations}
We recall some basic notions on $A_\infty$-deformations of graded algebras. We refer to \cite{Kel01} for  basic notions on $A_\infty$-algebras.
\begin{defn}
Let $(A, \mu)$ be a graded algebra with multiplication $\mu$. An {\it $A_\infty$-deformation} of $A$ is a unital $A_\infty$-algebra $(A, \mu_1, \mu_2, \mu_3, \ldots)$ such that $\mu_1=0$ and $\mu_2 = \mu$. Two $A_\infty$-deformations of $A$ are {\it equivalent} if they are $A_\infty$-isomorphic to each other. 

We say that $A$ is {\it intrinsically formal} if it admits no  nontrivial $A_\infty$-deformations. 
\end{defn}

Notice that a graded algebra $A$ may be viewed as a bigraded algebra by assigning each element in $A^i$ the bidegree $(i, -i)$. The following useful criterion for intrinsic formality will be used. 
\begin{thm}[{\cite[Cor.~4]{Kad}, \cite[Thm.~4.7]{ST} and \cite[Cor.~5.4]{RW}}] \label{thm:KST}  
Let $A$ be a graded algebra. If we have $\HH^{2,q}(A, A) = 0$ for all positive integers $q$, then $A$ is intrinsically formal. 
\end{thm}

\subsection{Zigzag algebras} We are interested in $A_\infty$-deformations of zigzag algebras.
\begin{defn}[{\cite[\S 3]{HK}}]\label{defn:zigzag}
Let $\Gamma$ be a finite connected graph without loops and multiple edges. If $\Gamma$ has more than two vertices, then the {\it zigzag algebra} $Z(\Gamma)$ is the path algebra $\Bbbk \overline \Gamma$ of the double quiver $\overline \Gamma$ (i.e.\ replacing each edge by two arrows with opposite directions) modulo the relations: all $2$-cycles at each vertex are equal (but nonzero) and all paths of length $2$ excluding $2$-cycles are zero. If the graph $\Gamma$ is $A_1$,  then by convention we define $Z(\Gamma) = \Bbbk [x]/(x^2)$. If the graph $\Gamma$ is $A_2$, then we define $Z(\Gamma)$ as the quotient of $\Bbbk \overline \Gamma$ by the two-sided ideal generated by all paths of length greater than $2$. 
\end{defn}
We view $Z(\Gamma)$ as a graded algebra by assigning each arrow degree $1$ and study its \linebreak $A_\infty$-deformations. Notice that $Z(\Gamma)$ is of finite total dimension and concentrated in degrees $0, 1$ and $2$ such that we have $\dim Z(\Gamma)^0 = |\Gamma_0| = \dim Z(\Gamma)^2$ and $\dim Z(\Gamma)^1 = 2 |\Gamma_1|$. For example, let $\Gamma= D_4$. Then we have $Z(\Gamma) = \Bbbk \overline\Gamma / (\alpha_i\beta_i=\alpha_j\beta_j, \beta_i\alpha_j=0)_{i\neq j}$, where $\overline\Gamma$ is given as follows. 
 
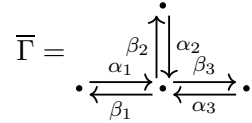
\[
\begin{tikzpicture}
\begin{scope}
\node at (-1.6,-0.55) {$\overline \Gamma=$};
\node[] (a0) at (-1.1,-1.1) {};
\filldraw [black] (-1.1,-1.1) circle (1pt);
\node[] (a1) at (0,-1.1) {};
\filldraw [black] (0,-1.1) circle (1pt);
\node[] (a2) at (0,0) {};
\filldraw [black] (0,0) circle (1pt);
\node[] (a3) at (1.1,-1.1) {};
\filldraw [black] (1.1,-1.1) circle (1pt);

\path[->, line width=.7pt] (a0) edge[transform canvas={yshift=.5ex}] node[font=\scriptsize,above=-.3ex] {$\alpha_1$} (a1);
\path[<-, line width=.7pt] (a0) edge[transform canvas={yshift=-.5ex}] node[font=\scriptsize,below=-.3ex] {$\beta_1$} (a1);
\path[<-, line width=.7pt] (a1) edge[transform canvas={xshift=.5ex}] node[font=\scriptsize,right=-.3ex] {$\alpha_2$} (a2);
\path[->, line width=.7pt] (a1) edge[transform canvas={xshift=-.5ex}] node[font=\scriptsize,left=-.3ex] {$\beta_2$} (a2);
\path[<-, line width=.7pt] (a1) edge[transform canvas={yshift=-.5ex}] node[font=\scriptsize,below=-.3ex] {$\alpha_3$} (a3);
\path[->, line width=.7pt] (a1) edge[transform canvas={yshift=.5ex}] node[font=\scriptsize,above=-.3ex] {$\beta_3$} (a3);
\end{scope}
\end{tikzpicture}
\]

\subsection{Derived Koszul duality between zigzag algebras and $2$-dimensional Ginzburg dg algebras}\label{subsection:3.3}

Let $\Gamma$ be a finite tree and $Q$ any fixed quiver whose underlying graph is $\Gamma$ and each vertex is a sink or a source. Recall from section \ref{subsection:hochschildginzburg} that $\Pi_2(Q)$ is bigraded. Similarly, the algebra $Z(\Gamma)$ is bigraded whose arrows are of bidegree $(1, -1)$. 

It follows from \cite[\S~5.3]{EL} that $Z(\Gamma)$ is derived Koszul dual to $\Pi_2(Q)$, i.e.\ there are quasi-isomorphisms
\begin{align}\label{algin:koszulduality}
\mathbb{R}\mathrm{Hom}^{\bullet,\bullet}_{Z(\Gamma)}(\Bbbk \Gamma_0, \Bbbk \Gamma_0) \simeq \Pi_2(Q)\quad \text{and} \quad \mathbb{R}\mathrm{Hom}^{\bullet,\bullet}_{\Pi_2(Q)}(\Bbbk Q_0, \Bbbk Q_0) \simeq Z(\Gamma)
\end{align}
of differential bigraded algebras. We remark that if we view $Z(\Gamma)$ as an ordinary graded algebra without Adams grading, then the dg algebra $\mathbb{R}\mathrm{Hom}_{Z(\Gamma)}(\Bbbk \Gamma_0, \Bbbk \Gamma_0)$ is quasi-isomorphic to the completion of $\Pi_2(Q)$  with respect to the length filtration.

The following result will be used in our proof. 
\begin{prop}[{\cite[Thm.~27]{EL} and \cite[\S~3.1]{Kel03}}] \label{prop:Zigginzburghh}
For any integers $p$ and $q$, we have an isomorphism 
\begin{equation}\label{isomorphism}
\begin{tikzcd}
\HH^{p,q}(Z(\Gamma), Z(\Gamma)) \arrow[no head]{r}{\sim} & \HH^{p, q}(\Pi_2(Q), \Pi_2(Q))\: .
\end{tikzcd}
\end{equation}
\end{prop}

\begin{rem}\label{rem:Koszulalgebra}
The above derived Koszul duality \eqref{algin:koszulduality} might not work if $\Gamma$ is not a tree. Nevertheless, recall from \cite{Mar} that for any graph $\Gamma$ as in Definition \ref{defn:zigzag}, $Z(\Gamma)$ is always Koszul (in the classical sense) if $\Gamma$ is not of type $ADE$, and its Koszul dual $Z(\Gamma)^!$ is the path algebra $\Bbbk \overline \Gamma$ modulo the relations: the sum of all $2$-cycles is zero; see \cite[\S~6.1]{HK}. Then by \cite[\S~3.5]{Kel03}, we have  
\[
\begin{tikzcd}
\HH^{p,q}(Z(\Gamma), Z(\Gamma)) \arrow[no head]{r}{\sim} & \HH^{p, q}(Z(\Gamma)^!, Z(\Gamma)^!)\: ,
\end{tikzcd} 
\]
where the algebra $Z(\Gamma)^!$ is bigraded whose arrows are of bidegree $(0, 1)$. If the graph $\Gamma$ is a tree and not of type $ADE$, then this coincides with the isomorphism~\eqref{isomorphism} since we have
\[
\begin{tikzcd}
Z(\Gamma)^! \arrow[no head]{r}{\sim} & \Lambda_Q & \Pi_2(Q) \arrow[swap]{l}{\sim} \: ;
\end{tikzcd}
\]
see \cite{HK} for the first isomorphism. 
\end{rem}

\subsection{The proof}\label{subsection4.3}
Now let us give our proof of Theorem \ref{thm:ET}. 
\begin{proof}[Another proof of Theorem \ref{thm:ET}]
Let us first prove the sufficiency. By Theorem \ref{thm:KST}, it suffices to show that if $\Gamma$ is of type $ADE$ and $\Bbbk$ is not of bad characteristic, then we have \linebreak $\HH^{2, q}(Z(\Gamma), Z(\Gamma)) =0 $ for all positive integers $q$. This follows by combining Proposition  \ref{prop:Zigginzburghh} and Corollary \ref{corollary:HH}.

Now we prove the necessity. For this, by Corollary~\ref{corollary:HH} and Proposition~\ref{prop:ginzburghh}, it suffices to show that any nonzero element $w$ in $(\Lambda_Q/[\Lambda_Q, \Lambda_Q])^{q+2}$ induces a nontrivial $A_\infty$-deformation of $Z(\Gamma)$. The element $w$ lifts to a linear combination $w'$ of cycles in $\Lambda_Q$ of length greater than $2$, which induces a nontrivial first-order dg deformation of $\Pi_2(Q)$ by $d'(t) = d(t) + w'\epsilon$, where $t = \sum_{i\in Q_0} t_i$. Notice that this deformation naturally extends to an actual (or global) dg deformation
\[
B'=(\Pi_2(Q), d'|_{\epsilon=1})
\]
of $\Pi_2(Q)$. Then the $A_\infty$-algebra $Z(\Gamma)'=\mathrm{Ext}^\bullet_{B'}(\Bbbk Q_0, \Bbbk Q_0)$ is a nontrivial $A_\infty$-deformation (compare \cite[Prop.~4.7]{BW}) of $Z(\Gamma)$ such that  $\mathbb{R}\mathrm{Hom}_{Z(\Gamma)'}(\Bbbk \Gamma_0, \Bbbk \Gamma_0)$ is quasi-isomorphic to the completion of $B'$ with respect to the length filtration. Here we need to take the completion since $B'$ and $Z(\Gamma)'$ are ordinary dg algebras without Adams grading; compare the derived Koszul duality~\eqref{algin:koszulduality}. 
\end{proof}

\begin{rem}
Theorem \ref{thm:ET} may be more general. Namely, let $\Gamma$ be any graph as in Definition \ref{defn:zigzag} which is not of type $ADE$. Then the graded algebra $\Pi_2(\Gamma)$ is not intrinsically formal, i.e.\ it admits nontrivial $A_\infty$-deformations. The reason is as follows.  By Remark \ref{rem:Koszulalgebra}, we have
\begin{align*}
\HH^{2, q}(Z(\Gamma), Z(\Gamma))&\simeq \HH^{2, q}(Z(\Gamma)^!, Z(\Gamma)^!) \\
&\simeq \HH_{0, q}(Z(\Gamma)^!,Z(\Gamma)^!) \\
&\simeq (Z(\Gamma)^!/[Z(\Gamma)^!,Z(\Gamma)^!])^{0, q+2}\: ,
\end{align*}
where the second isomorphism follows since $Z(\Gamma)^!$ is left $2$-Calabi--Yau (as $Z(\Gamma)$ is right $2$-Calabi--Yau). We point out that $Z(\Gamma)^!$ might not be isomorphic to a preprojective algebra unless $\Gamma$ is a tree (or bipartite); compare Remark \ref{rem:Koszulalgebra}. But by a similar computation to \cite[Thm.~13.1.1]{Sc} and \cite[Theorem 1.4.b]{MOV}, also see \cite[page 518]{MOV},  we may show that $(Z(\Gamma)^!/[Z(\Gamma)^!,Z(\Gamma)^!])^{0, >2}$ is of infinite total dimension.  \end{rem}

\begin{exm}
Let $\Gamma$ be the extended D$_4$ graph.   Then combining Propositions \ref{prop:ginzburghh} and \ref{prop:Zigginzburghh} with  \cite[Thm.~13.1.1]{Sc} we obtain $\HH^{2, >0}(Z(\Gamma), Z(\Gamma)) \simeq (e_4 \Lambda_Q e_4)^{>2}$, which is of infinite total dimension. Here $e_4$ corresponds to the bottom vertex of the quiver below. So $Z(\Gamma)$ admits an infinite-dimensional family of $A_\infty$-deformations. For instance, the cycle $\beta_4\alpha_1\beta_1\alpha_4$ induces a nontrivial $A_\infty$-deformation of $Z(\Gamma)$ such that $m_4$ is nonzero only when
\begin{align*}
m_4(\beta_4 \otimes \alpha_1 \otimes \beta_1 \otimes \alpha_4) = \beta_4 \alpha_4 \ko & m_4(\alpha_4 \otimes \beta_4 \otimes \alpha_1 \otimes \beta_1) = \alpha_1 \beta_1 \: , \\
m_4(\beta_1 \otimes \alpha_4 \otimes \beta_4 \otimes \alpha_1) = \beta_1 \alpha_1 \ko & m_4(\alpha_1 \otimes \beta_1 \otimes \alpha_4 \otimes \beta_4) = \alpha_4 \beta_4
\end{align*}
(up to scalar) and all the other higher multiplications acting on arrows vanish. 
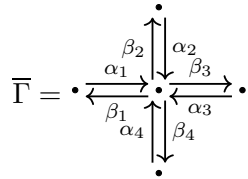
\[
\begin{tikzpicture}
\begin{scope}
\node at (-1.6,0) {$\overline \Gamma=$};
\node[] (a0) at (-1.1,0) {};
\filldraw [black] (-1.1,0) circle (1pt);
\node[] (a1) at (0,0) {};
\filldraw [black] (0,0) circle (1pt);
\node[] (a2) at (0,1.1) {};
\filldraw [black] (0,1.1) circle (1pt);
\node[] (a3) at (1.1,0) {};
\filldraw [black] (1.1,0) circle (1pt);
\node[] (a4) at (0,-1.1) {};
\filldraw [black] (0,-1.1) circle (1pt);

\path[->, line width=.7pt] (a0) edge[transform canvas={yshift=.5ex}] node[font=\scriptsize,above=-.3ex] {$\alpha_1$} (a1);
\path[<-, line width=.7pt] (a0) edge[transform canvas={yshift=-.5ex}] node[font=\scriptsize,below=-.3ex] {$\beta_1$} (a1);
\path[<-, line width=.7pt] (a1) edge[transform canvas={xshift=.5ex}] node[font=\scriptsize,right=-.3ex] {$\alpha_2$} (a2);
\path[->, line width=.7pt] (a1) edge[transform canvas={xshift=-.5ex}] node[font=\scriptsize,left=-.3ex] {$\beta_2$} (a2);
\path[<-, line width=.7pt] (a1) edge[transform canvas={yshift=-.5ex}] node[font=\scriptsize,below=-.3ex] {$\alpha_3$} (a3);
\path[->, line width=.7pt] (a1) edge[transform canvas={yshift=.5ex}] node[font=\scriptsize,above=-.3ex] {$\beta_3$} (a3);
\path[<-, line width=.7pt] (a1) edge[transform canvas={xshift=-.5ex}] node[font=\scriptsize,left=-.3ex] {$\alpha_4$} (a4);
\path[->, line width=.7pt] (a1) edge[transform canvas={xshift=.5ex}] node[font=\scriptsize,right=-.3ex] {$\beta_4$} (a4);
\end{scope}
\end{tikzpicture}
\]

\end{exm}
\vskip -5pt

\noindent {\bf Acknowledgements.}  \; The authors are very grateful to Severin Barmeier, Bernhard Keller, Yank\i~Lekili and Travis Schedler for many helpful comments and discussions. The first-named author also  warmly thanks the Institute of Algebra and Number Theory at the University of Stuttgart for the hospitality while working on this paper. The second-named author was supported by the National Key R\&D Program of China (2024YFA1013803), the NSFC \linebreak (Nos.~13004005, 12371043 and 12071137) and the DFG grant (WA 5157/1-1). The authors are grateful to the referee for helpful comments and suggestions.

\vskip 7pt

{\footnotesize \noindent Junyang Liu\\  School of Mathematical Sciences, University of Science and Technology of China \\Hefei 230026, China \\liuj$\symbol{64}$imj-prg.fr \\https://webusers.imj-prg.fr/\~{}junyang.liu/}

\vskip 5pt 

{\footnotesize \noindent Zhengfang Wang\\School of Mathematics, Nanjing University\\Nanjing 210093, China \\zhengfangw$\symbol{64}$gmail.com}

\end{document}